\documentclass[11pt, notitlepage]{article}   

\usepackage{amsmath,amsthm,amsfonts}   

\usepackage{amscd}
\usepackage{amsfonts}
\usepackage{amssymb}
\usepackage{color}      
\usepackage{epsfig}
\usepackage{graphicx}           

\theoremstyle{plain}
\newtheorem{Thm}{Theorem}[section]
\newtheorem{Lem}[Thm]{Lemma}
\newtheorem{Crlr}[Thm]{Corollary}
\newtheorem{Prop}[Thm]{Proposition}

\newtheorem{Rem}[Thm]{Remark} 

\theoremstyle{definition}
\newtheorem{Def}[Thm]{Definition}

\theoremstyle{remark}

\def\finf{\mathop{{\rm I}\kern -.27 em {\rm F}}\nolimits}

\begin{document}

\title{On Metric Dimension of Functigraphs}

\author{{\bf Linda Eroh$^1$}, {\bf Cong X. Kang$^2$}, and {\bf Eunjeong Yi$^3$}\\
\small$^1$ University of Wisconsin Oshkosh, Oshkosh, WI 54901, USA\\
\small $^{2,3}$Texas A\&M University at Galveston, Galveston, TX 77553, USA\\
$^1${\small\em eroh@uwosh.edu}; $^2${\small\em kangc@tamug.edu}; $^3${\small\em yie@tamug.edu}}

\maketitle

\date{}

\begin{abstract}
The \emph{metric dimension} of a graph $G$, denoted by $\dim(G)$, is the minimum number of vertices such that each vertex is uniquely determined by its distances to the chosen vertices. Let $G_1$ and $G_2$ be disjoint copies of a graph $G$ and let $f: V(G_1) \rightarrow V(G_2)$ be a function. Then a \emph{functigraph} $C(G, f)=(V, E)$ has the vertex set $V=V(G_1) \cup V(G_2)$ and the edge set $E=E(G_1) \cup E(G_2) \cup \{uv \mid v=f(u)\}$. We study how metric dimension behaves in passing from $G$ to $C(G,f)$ by first showing that $2 \le \dim(C(G, f)) \le 2n-3$, if $G$ is a connected graph of order $n \ge 3$ and $f$ is any function. We further investigate the metric dimension of functigraphs on complete graphs and on cycles.
\end{abstract}


\vspace{.05in}



\section{Introduction}

Let $G = (V(G),E(G))$ be a simple, undirected, connected, and nontrivial graph with order $|V(G)|$. The \textit{degree} of a vertex $v$ in $G$, denoted by $\deg_G(v)$, is the number of edges that are incident to $v$ in $G$; an \emph{end-vertex} is a vertex of degree one, and a \emph{support vertex} is a vertex that is adjacent to an end-vertex. For a vertex $v \in V(G)$, the \emph{open neighborhood of $v$} is the set $N_G(v)=\{u \mid uv \in E(G)\}$, and the \emph{closed neighborhood of $v$} is the set $N_G[v]=N_G(v) \cup \{v\}$. For $S \subseteq V(G)$, the \emph{open neighborhood of $S$} is the set $N_G(S)=\cup_{v \in S} N_G(v)$ and the \emph{closed neighborhood of $S$} is the set $N_G[S]=N_G(S) \cup S$; throughout the paper, we denote by $N(S)$ ($N[S]$, respectively) the open (closed, respectively) neighborhood of $S$ in $C(G, f)$. We denote by $K_n$, $C_n$, and $P_n$ the complete graph, the cycle, and the path on $n$ vertices, respectively. The \emph{distance} between two vertices $v, w \in V(G)$ is denoted by $d_G(v,w)$; throughout the paper, we denote by $d(v,w)$ the distance between $v$ and $w$ in $C(G,f)$. For an ordered set $S=\{u_1, u_2, \ldots, u_k\} \subseteq V(G)$ of distinct vertices, the \emph{metric code} (or \emph{code}, in short) of $v \in V(G)$ with respect to $S$ is the $k$-vector $code_S(v)=(d_G(v, u_1), d_G(v, u_2), \ldots, d_G(v, u_k))$. A vertex $x \in V(G)$ \emph{resolves} a pair of vertices $v,w \in V(G)$ if $d_G(v,x) \neq d_G(w,x)$. A set of vertices $S \subseteq V(G)$ \emph{resolves} $G$ if every pair of distinct vertices of $G$ are resolved by some vertex in $S$ or, equivalently, if $code_S(u) \neq code_S(v)$ for distinct vertices $u$ and $v$ of $G$; then $S$ is called a \emph{resolving set} of $G$. The \emph{metric dimension} of $G$, denoted by $\dim(G)$, is the minimum of $|S|$ as $S$ varies over all resolving sets of $G$. For other terminologies in graph theory, we refer to \cite{CZ}.\\

Slater  \cite{Slater, Slater2} introduced the concept of a resolving set for a connected graph under the term \emph{locating set}; he referred to a minimum resolving set as a \emph{reference set}, and the cardinality of a minimum resolving set as the \emph{location number} of a graph. Independently, Harary and Melter \cite{HM} studied these concepts under the term \emph{metric dimension}. Metric dimension as a graph parameter has numerous applications, among them are robot navigation \cite{landmarks}, sonar \cite{Slater}, combinatorial optimization \cite{MathZ}, and pharmaceutical chemistry \cite{CEJO}. It was noted in \cite{NPcompleteness} that determining the metric dimension of a graph is an NP-hard problem. Metric dimension has been heavily studied; for surveys, see \cite{base} and \cite{survey}. For more articles on metric dimension in graphs, see \cite{ref1}, \cite{cartesian}, \cite{CEJO}, \cite{CPZ}, \cite{EKY}, \cite{ref2}, \cite{ref3}, \cite{landmarks}, \cite{PoZh}, and \cite{wheels}.\\

Chartrand and Harary \cite{permutation} introduced a ``permutation graph" (or ``generalized prism"). Hedetniemi \cite{H} introduced a ``function graph", which comprises two graphs (not necessarily identical copies) with a function relation between them. Independently, D\"{o}rfler \cite{WD} introduced a ``mapping graph", which consists of two disjoint identical copies of a graph and additional edges between the two vertex sets specified by a function. The ``mapping graph" was rediscovered and studied in \cite{functi}, where it was called a ``functigraph". We recall the definition of the functigraph.

\begin{Def}
Let $G_1$ and $G_2$ be disjoint copies of a graph $G$, and let $f: V(G_1) \rightarrow V(G_2)$ be a function. A \emph{functigraph} $C(G, f)=(V, E)$ consists of the vertex set $V=V(G_1) \cup V(G_2)$ and the edge set $E(G)=E(G_1) \cup E(G_2) \cup \{uv \mid v=f(u)\}$.
\end{Def}

In this paper, we study the metric dimension of functigraphs. For a connected graph $G$ of order $n\ge 3$ and for a function $f$, we show that $2 \le \dim(C(G,f)) \le 2n-3$. We provide in Remark \ref{remark} an example showing that $\dim(G)-\dim(C(G,f))$ can be arbitrarily large; we also show in Theorem \ref{thmcycle_constant} that $\dim(C(C_n,f))$ can be arbitrarily large for a constant function $f$, though $\dim(C_n)=2$. These examples are quite surprising; they indicate the complexity, vis-\`{a}-vis metric dimension, present in passing from $G$ to $C(G,f)$. Further, we give the metric dimension of functigraphs on complete graphs and we also give bounds for the metric dimension of functigraphs on cycles. It is worth noting that the metric dimension of the wheel graph, which is a subgraph of $C(C_n, f)$ for a constant function $f$, has been studied in \cite{wheel2} and \cite{wheels}.
 

\section{Bounds on Metric Dimension of Functigraphs}

We first recall some basic facts on metric dimension for background.

\begin{Thm} \cite{CEJO} \label{dimbounds}
For a connected graph $G$ of order $n \ge  2$ and diameter $d$, $$f (n, d) \le \dim(G) \le n-d,$$
where $f (n, d)$ is the least positive integer $k$ for which $k + d^k \ge n$.
\end{Thm}

A generalization of Theorem \ref{dimbounds} has been given in \cite{EJC} by Hernando et al.

\begin{Thm} \cite{EJC}
Let $G$ be a graph of order $n$, diameter $d \ge 2$, and metric dimension $k$. Then
$$n \le \left(\left\lfloor \frac{2d}{3}\right\rfloor+1\right)^k+k\sum_{i=1}^{\lceil \frac{d}{3} \rceil} (2i-1)^{k-1}.$$
\end{Thm}

\begin{Thm} \cite{CEJO} \label{dimthm}
Let $G$ be a connected graph of order $n \ge 2$. Then
\begin{itemize}
\item[(a)] $\dim(G)=1$ if and only if $G=P_n$,
\item[(b)] $\dim(G)=n-1$ if and only if $G=K_n$,
\item[(c)] for $n \ge 4$, $\dim(G)=n-2$ if and only if $G=K_{s,t}$ ($s,t \ge 1$), $G=K_s + \overline{K}_t$ ($s \ge 1, t \ge 2$), or $G=K_s + (K_1 \cup K_t)$ ($s, t \ge 1$); here, $A+B$ denotes the graph obtained from the disjoint union of graphs $A$ and $B$ by joining every vertex of $A$ with every vertex of $B$, and $\overline{C}$ denotes the complement of a graph $C$.
\end{itemize}
\end{Thm}

Next, we obtain general bounds for the metric dimension of functigraphs. If $G$ is a connected graph of order 2, then $G \cong P_2$ and $\dim(C(P_2,f))=2$ for any function $f$. So, we only consider a connected graph $G$ of order $n \ge 3$ for the rest of the paper.

\begin{Thm}\label{functibounds}
Let $G$ be a connected graph of order $n \ge 3$, and let $f: V(G_1) \rightarrow V(G_2)$ be a function. Then $2 \le \dim(C(G, f)) \le 2n-3$. Both bounds are sharp.
\end{Thm}

\begin{proof}
Since $C(G, f)$ contains a cycle, $\dim(C(G, f)) \ge 2$ by (a) of Theorem \ref{dimthm}. On the other hand, noting that $C(G, f) \not\cong K_{2n}$ for any function $f$, $\dim(C(G, f)) \le 2n-2$ by (b) of Theorem \ref{dimthm}; further, no $C(G, f)$ satisfies $\dim(C(G, f))=2n-2$ by (c) of Theorem \ref{dimthm}, and thus $\dim(C(G, f)) \le 2n-3$. For the sharpness of the lower bound, take $G=P_n$ and $f \equiv id$, the identity function; then $\dim(C(P_n, id))=2$, since two end-vertices of $G_1 \cong P_n$ form a minimum resolving set for $C(P_n, id)$. For the sharpness of the upper bound, $G=K_n$, with $f$ a constant function, is an example -- this will be shown in Theorem \ref{constant}. 
\end{proof}

The following definitions are stated in~\cite{CEJO}. Fix a graph $G$. A vertex of degree at least three is called a \emph{major vertex}. An end-vertex $u$ is called \emph{a terminal vertex of a major vertex} $v$ if $d_G(u, v)<d_G(u, w)$ for every other major vertex $w$. The \emph{terminal degree} of a major vertex $v$ is the number of terminal vertices of $v$. A major vertex $v$ is an \emph{exterior major vertex} if it has positive terminal degree. Let $\sigma(G)$ denote the sum of terminal degrees of all major vertices of $G$, and let $ex(G)$ denote the number of exterior major vertices of $G$. Next, we recall the concept of \emph{twin vertices}. Two vertices $u, v \in V(G)$ are called \emph{twins} if $N_G(u) - \{v\}=N_G(v) - \{u\}$; notice that for any set $S$ with $S \cap \{u,v\} = \emptyset$, $code_S(u)=code_S(v)$.\\

\begin{Thm}\cite{CEJO, landmarks, PoZh} \label{tree}
If $T$ is a tree that is not a path, then $\dim(T)=\sigma(T)-ex(T)$.
\end{Thm}

\begin{Rem}\label{remark}
The graph in Figure \ref{md-counter} shows that $\dim(G)-\dim(C(G,f))$ can be arbitrarily large. Let $f_i$ denote the restriction of $f$ on $V(G_1^i)$, where $1 \le i \le n$. Notice $\dim(G_1^i)=\sigma(G_1^i)-ex(G_1^i)=6-1=5$. But $\dim(C(G_i^i, f_i)) \le 4$, since the solid vertices in $C(G_i^i, f_i)$ form a resolving set for $C(G_1^i, f_i)$, as one can explicitly check. Also, notice that, for any $u, v \in V(C(G_1^i, f_i))$, the distance between $u$ and $v$ in $C(G_1, f)$ is not less than their distance in $C(G_1^i, f_i)$; i.e., the solid vertices of $C(G_1^i, f_i)$ still distinguish the vertices of $C(G_1^i, f_i)$ in $C(G_1, f)$. Thus, $\dim(C(G_1, f)) \le n \cdot \dim(C(G_1^i, f_i)) \le 4n$. But, $\dim(G_1)=n \cdot \dim(G_1^i)=5n$. Therefore, $\dim(G_1) - \dim(C(G_1, f)) \ge n$.
\end{Rem}

\begin{figure}[htpb]
\begin{center}
\scalebox{0.4}{\input{FM.pstex_t}} \caption{An example showing that $\dim(G)-\dim(C(G,f))$ can be arbitrarily large}\label{md-counter}
\end{center}
\end{figure}


\section{Metric Dimension of Functigraphs on Complete Graphs}

In this section, for $n \ge 3$, we show that 1) $\dim(C(K_n, \sigma))=n-1$ for a permutation $\sigma$; 2) $\dim(C(K_n, f))=2n-2-|f(V(G_1))|$ for a non-permutation function $f$. Throughout this section, we let $V(G_1)=\{u_i \mid 1 \le i \le n\}$ and $V(G_2)=\{v_i \mid 1 \le i \le n\}$ for $G_1 \cong G_2 \cong K_n$, where $n \ge 3$. Noting that each vertex in a resolving set $S$ has a distinct code from any other vertex in $C(G,f)$, one only needs to check the codes of vertices in $V(C(G, f)) - S$.\\

We recall the following theorem, which can be viewed as a functigraph where the function is the identity on a connected graph $G$.

\begin{Thm}\cite{CEJO}\label{id}
For every connected graph $G$, $\dim(G) \le \dim(G \square K_2) \le \dim(G)+1$, where $A \square B$ denotes the Cartesian product of two graphs $A$ and $B$.
\end{Thm}

\begin{Thm}\cite{cartesian} \label{siamdm}
For every graph $G$ and for all $n \ge 2 \dim(G)+1$, $\dim(K_n \square G)=n-1$.
\end{Thm}

As an immediate corollary of Theorem \ref{siamdm}, we have the following

\begin{Crlr}
Let $G=K_n$ be the complete graph of order $n \ge 3$, and let $\sigma: V(G_1) \rightarrow V(G_2)$ be a permutation. Then $\dim(C(K_n, \sigma))=n-1$.
\end{Crlr}

\begin{Thm}\label{constant}
Let $G=K_n$ be the complete graph of order $n \ge 3$, and let $f_0: V(G_1) \rightarrow V(G_2)$ be a constant function. Then $\dim(C(K_n, f_0))=2n-3$.
\end{Thm}

\begin{proof}
Without loss of generality, let $f_0(u_i)=v_1$ for each $i$ ($1 \le i \le n$). Let $S$ be a minimum resolving set of $C(K_n, f_0)$. First, we will show that $|S| \ge 2n-3$ for $n \ge 3$. Since any two vertices in $G_1$ are twins, all but one of these $n$ vertices must be in $S$. Similarly, since any two vertices in $\{v_2, v_3, \ldots, v_n\}$ are twins, all but one of these ($n-1$) vertices must be in $S$. So, $|S| \ge (n-1)+(n-2)=2n-3$. By Theorem \ref{functibounds}, $|S| \le 2n-3$, and thus $\dim(C(K_n, f_0))=2n-3$. More explicitly, one can check that $S=\{u_1, u_2, \ldots, u_{n-1}, v_2, v_3, \ldots, v_{n-1}\}$ is a resolving set for $C(K_n, f_0)$. \hfill
\end{proof}
\begin{figure}[htpb]
\begin{center}
\scalebox{0.47}{\input{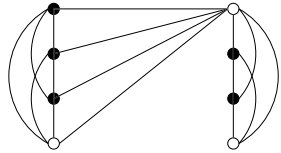}} \caption{$\dim(C(K_4, f_0))=5$ for a constant function $f_0$}\label{completeC}
\end{center}
\end{figure}

\begin{Rem}
Theorem \ref{constant} may be generalized as follows. For $m, n \ge 3$, let $H_1=K_m$ and $H_2=K_n$. Let $f_0:V(H_1) \rightarrow V(H_2)$ be a constant function. Let $\mathcal{G}=(V,E)$ be the graph with $V=V(H_1) \cup V(H_2)$ and $E=E(H_1) \cup E(H_2) \cup \{u_iv_1 \mid f_0(u_i)=v_1 \mbox{ for each $i$ } (1 \le i \le m) \}$. Then $\dim(\mathcal{G})=m+n-3$.
\end{Rem}

\begin{Thm}
Let $G=K_n$ be the complete graph of order $n \ge 3$, and let $|f(V(G_1))|=s$ where $1 < s < n$. Then $\dim(C(K_n, f)) = 2n-2-s$. 
\end{Thm}

\begin{proof}
Without loss of generality, we may assume that $f(V(G_1))=\{v_1, v_2, \ldots, v_s\}$ such that $|f^{-1}(v_i)|=k_i$ with $k_1 \ge k_2 \ge  \ldots \ge k_s \ge 1$ and $\sum_{i=1}^{s}k_i=n$. Further, we may assume that $f^{-1}(v_1)=\{u_i \mid 1 \le i \le k_1\}$, $f^{-1}(v_2)=\{u_i \mid k_1+1 \le i \le k_1+k_2\}, \ldots$, and $f^{-1}(v_s)=\{u_i \mid 1+ \sum_{t=1}^{s-1} k_{t} \le i \le \sum_{t=1}^{s} k_{t}\}$; we adopt the convention that $\sum_{i=a}^{b} f(i)=0$ when $b<a$. Since $|f(V(G_1))|<n$, $k_1 \ge 2$. 

\vspace{.1in}

First, we show that $\dim(C(K_n, f)) \ge 2n-2-s$. Let $S$ be any minimum resolving set of $C(K_n,f)$. Since any two vertices in $\{v_{s+1}, v_{s+2}, \ldots, v_n\}$ are twins, all but one of these ($n-s$) vertices must be in $S$, say $S_0=\{v_{s+1}, v_{s+2}, \ldots, v_{n-1}\} \subseteq S$ with $|S_0|=n-s-1$. Similarly, for each $i$ ($1 \le i \le s$), $f^{-1}(v_i)$ consists of $k_i$ vertices such that any two vertices in $f^{-1}(v_i)$ are twins, and thus ($k_i-1$) of each $f^{-1}(v_i)$ must be in $S$; we may assume that $S_1=\cup_{i=1}^{s} \{u_j \mid 1+\sum_{t=1}^{i-1}k_t \le j \le (\sum_{t=1}^{i}k_t)-1\} \subseteq S$ with $|S_1|=n-s$. We denote by $a_i=\sum_{t=1}^{i}k_t$ for each $i$ ($1 \le i \le s$). For each $i$ and $j$, $1 \le i < j \le s$, consider the two vertices $u_{a_i}$ and $u_{a_j}$. They are both distance 1 from every $u_{\ell}$ with $\ell \neq a_i$ and $\ell \neq a_j$ and distance 2 from every $v_{m}$ with $m \neq i$ and $m \neq j$. Thus, to resolve these two vertices, one of the vertices in $\{u_{a_i}, u_{a_j}, v_i, v_j\}$ must be in $S$. Since this is true for every pair $i$ and $j$ with $1 \le i < j \le s$, there exists at most one $\ell$ such that $\{u_{a_{\ell}}, v_{\ell}\} \cap S=\emptyset$. Thus, we have $|S| \ge (n-s-1)+(n-s)+(s-1)=2n-2-s$.

\vspace{.1in}

Next, we show that $S=\{u_2, u_3, \ldots, u_n, v_{s+1}, v_{s+2}, \ldots, v_{n-1}\}$ is a resolving set of $C(K_n, f)$ with $|S|=(n-1)+(n-1-s)=2n-2-s$. Clearly, (i) $u_1$ has 1 in the $k$-th entry and 2 in the rest of the entries of its code, where $1 \le k \le n-1$; (ii) for $1 \le i \le s$, each $v_i$ has 1 in the $k$-th entry and 2 in the rest of the entries of its code, where $n \le k \le 2n-2-s$ or $\max\{1, \sum_{t=1}^{i-1}k_t\} \le k \le (\sum_{t=1}^{i}k_t)-1$; (iii) $v_n$ has 2 in the $k$-th entry and 1 in the rest of the entries of its code, where $1 \le k \le n-1$. Thus, $S$ is a resolving set for $C(K_n, f)$, and hence $\dim(C(K_n, f)) \le 2n-2-s$ for $1<s<n$. 

\vspace{.1in}

Therefore, $\dim(C(K_n,f))=2n-2-s$ for $1 < s <n$. \hfill
\end{proof}
\begin{figure}[htpb]
\begin{center}
\scalebox{0.43}{\input{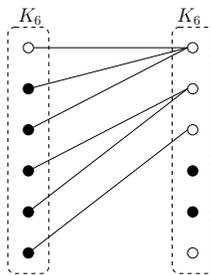}} \caption{$\dim(C(K_6, f)) = 7$, where $|f(V(G_1))|=3$}\label{completeF}
\end{center}
\end{figure}


\section{Metric Dimension of Functigraphs on Cycles}

In this section, we give bounds of the metric dimension of functigraphs on cycles. Let $G=C_n$ for $n \ge 3$, and let $G_1$ and $G_2$ be disjoint copies of $G$. Let $V(G_1)=\{u_i \mid 1 \le i \le n\}$ and let $E(G_1)=\{u_iu_{i+1} \mid 1 \le i \le n-1\} \cup \{u_1u_n\}$; similarly, let $V(G_2)=\{v_i \mid 1 \le i \le n\}$ and let $E(G_1)=\{v_iv_{i+1} \mid 1 \le i \le n-1\} \cup \{v_1v_n\}$. 

\begin{Prop}\cite{cartesian}\label{Cid}
For $n \ge 3$, 
\begin{equation*} \dim(C_n \square K_2) = \left\{
\begin{array}{ll}
2 & \mbox{ if $n$ is odd } \\
3 & \mbox{ if $n$ is even }  .
\end{array} \right.
\end{equation*}
\end{Prop}

\begin{Thm}
Let $G=C_n$ be the cycle of order $n \ge 3$, and let $\sigma: V(G_1) \rightarrow V(G_2)$ be a permutation. Then $2 \le \dim(C(C_n, \sigma)) \le n-1$, and both bounds are attainable. 
\end{Thm}

\begin{proof}
By Theorem \ref{functibounds}, $\dim(C(C_n, \sigma)) \ge 2$. Next, we show that $\dim(C(C_n,\sigma)) \le n-1$. We will show that $S=\{u_2, u_3, \ldots, u_n\}$ is a resolving set of $C(C_n, \sigma)$. Note that (i) $u_1$ has 1 exactly in the 1st and in the $(n-1)$th entries of its code; (ii) $\sigma(u_1)$ does not contain 1 in any entry of its code; (iii) for $2 \le i \le n$, each $\sigma(u_i)$ has 1 exactly in the $(i-1)$th entry of its code. Since $S$ is a resolving set of $C(C_n, \sigma)$ with $|S|=n-1$, $\dim(C(C_n,\sigma)) \le n-1$. For the sharpness of the lower bound, $\dim(C(C_5, id))=2$ by Proposition \ref{Cid}; for the sharpness of the upper bound, $\dim(C(C_4, id))=3$ by Proposition \ref{Cid}. \hfill
\end{proof}

We recall the following 

\begin{Thm} \cite{wheel2,wheels} \label{wheel}
For $n \ge 3$, let $W_{1,n}=C_n+K_1$ be the wheel graph on $n+1$ vertices. Then 
\begin{equation*} \dim(W_{1,n}) = \left\{
\begin{array}{cl}
3 & \mbox{if $n=3$ or $n=6$}, \\
\lfloor \frac{2n+2}{5}\rfloor & \mbox{otherwise} .
\end{array} \right.
\end{equation*}
\end{Thm}

In order to determine the metric dimension of $C(C_n, f_0)$ for a constant function $f_0$, we need the following lemma.

\begin{Lem} \label{lemC}
Let $G=C_n$ be the cycle of order $n \ge 6$, and let $f_0: V(G_1) \rightarrow V(G_2)$ satisfy $f_0(u_i)=v_1$ for each $i$ ($1 \le i \le n$). Let $S$ be a minimum resolving set of $C(C_n, f_0)$. 
\begin{itemize}
\item[(a)] There exists at most one vertex $u$ in $G_1$ such that $u \not\in N[S]$. Moreover, if $C_n$ is an even cycle and $|V(G_1) - N[S]|=1$, then at least two vertices in $G_2$ must belong to $S$.
\item[(b)] Let $S \cap V(G_1)=\{u_{i_1}, u_{i_2}, \ldots, u_{i_t}\}$, where $i_1<i_2< \ldots <i_t $. If $i_{j+1}-i_{j} = 3$ for some $j$, say $2 \le j \le n-2$ by relabeling if necessary, then $d_{G_1}(u_{i_j},u_{i_{j-1}}) \le 2$ and $d_{G_1}(u_{i_{j+2}},u_{i_{j+1}}) \le 2$.
\item[(c)] For each $u_i \in S \cap V(G_1)$, to resolve $A=N(u_i) \cap V(G_1)$, at least a vertex in $(N[A] - \{u_i\}) \cap V(G_1)$ must belong to $S$.
\end{itemize}
\end{Lem}

\begin{proof}
(a) Suppose there exist at least two vertices in $V(G_1) - N[S]$, say $u_p, u_q \in V(G_1) - N[S]$. Then, for each $x \in S \cap V(G_1)$, $d(u_p, x)=d(u_q, x)=2$. Noting that no vertex in $G_2$ resolves any two vertices in $G_1$, we have $code_S(u_p)=code_S(u_q)$. Thus at most one vertex belongs to $V(G_1) - N[S]$. Next, suppose that $u \in V(G_1) - N[S]$ and $C_n$ is an even cycle. Notice that, for each $x \in S \cap V(G_1)$, $d(u, x)=d(v_2, x)=d(v_n, x)=2$. Assume that $|S \cap V(G_2)|=1$. Since neither $v_1$ nor $v_{\frac{n}{2}+1}$ resolves $v_2$ and $v_n$, we may assume that $y \in S \cap V(G_2)$ for some $y \in \{v_2, v_3, \ldots, v_{\frac{n}{2}}\}$. But, $d(y, u)=d(y, v_1)+1=d(y, v_n)$ for $y \in \{v_2, v_3, \ldots, v_{\frac{n}{2}}\}$. Thus, $|S \cap V(G_2)| \ge 2$ in this case.\\

(b) Suppose that $i_{j+1}-i_{j}=3$ for some $j$ with $2 \le j \le n-2$ (by relabeling if necessary), and that $d_{G_1}(u_{i_j},u_{i_{j-1}}) \ge 3$ or $d_{G_1}(u_{i_{j+2}},u_{i_{j+1}}) \ge 3$. Without loss of generality, assume that  $d_{G_1}(u_{i_j},u_{i_{j-1}}) \ge 3$. Then the two vertices in $N(u_{i_j}) \cap V(G_1)$ have the same code.\\

(c) Let $u_i \in S \cap V(G_1)$, where $3 \le i \le n-2$ by relabeling if necessary. If $\{u_{i-2}, u_{i-1}, u_{i+1}, u_{i+2}\} \cap S=\emptyset$, then $code_S(u_{i-1})=code_S(u_{i+1})$.~\hfill 
\end{proof}

\begin{Thm}\label{thmcycle_constant}
Let $G=C_n$ be the cycle of order $n \ge 3$, and let $f_0:V(G_1) \rightarrow V(G_2)$ be a  constant function. Then
\begin{equation*} \dim(C(C_n,f_0)) = \left\{
\begin{array}{cl}
3 & \mbox{ if $n=3$}, \\
\left\lceil\frac{2n+3}{5}\right\rceil & \mbox{ if $n$ is odd and } n \neq 3, \\
\left\lceil\frac{2n}{5}\right\rceil + 1 & \mbox{ if $n$ is even} .
\end{array} \right.
\end{equation*}
\end{Thm}

\begin{proof}
Without loss of generality, let $f_0(u_i)=v_1$ for each $i$ ($1 \le i \le n$). Let $S$ be a minimum resolving set of $C(C_n, f_0)$, where $n \ge 3$. First, we consider $3 \le n \le 5$. Notice that no vertex of $G_2$ resolves any two vertices of $G_1$, so at least two vertices of $G_1$ must be in $S$. Furthermore, at least a vertex in $V(G_2) - \{v_1\}$ must belong to $S$, since $v_2$ and $v_n$ are not resolved by any vertex of $G_1$. So, $|S| \ge 3$ for $3 \le n \le 5$. One easily checks that $\{u_1, u_3, v_3\}$ is a resolving set for $C(C_3, f_0)$ and $C(C_5, f_0)$. It's also easily checked that $\{u_1, u_2, v_2\}$ is a resolving set for $C(C_4, f_0)$. Thus, $\dim(C(C_n, f_0))=3$, consistent with the formula asserted in our theorem, for $3 \le n \le 5$.\\

Next, we consider for $n \geq 6$. Notice that no vertex of $G_2$ resolves any two vertices of $G_1$, and, as already observed, at least a vertex of $V(G_2) - \{v_1\}$ must belong to $S$.

\vspace{.1in}

\textbf{Claim 1.} For $n \ge 6$, $|S| \ge \lceil \frac{2n+3}{5}\rceil$ if $n$ is odd and $|S| \ge \lceil \frac{2n}{5}\rceil+1$ if $n$ is even.\\

\textit{Proof of Claim 1.} Let $S \cap V(G_1) = \{u_{i_1}, u_{i_2}, \ldots, u_{i_t}\}$, where $i_1 < i_2 < \ldots < i_t$, with $|S \cap V(G_1)|=t$. We consider two cases.

\vspace{.1in}

\emph{Case 1. There is a vertex $u \in V(G_1)$ such that $u \not \in N[S] \cap V(G_1)$:} Without loss of generality, we may assume that $\{u_1, u_2, \ldots, u_{n-1}\} \subseteq N[S] \cap V(G_1)$ and $u_n \not \in N[S] \cap V(G_1)$.  Then $i_1=2$, $i_t =n-2$, and, by (b) of Lemma \ref{lemC}, we have $i_2-i_1\leq 2$ and $i_t - i_{t-1} \leq 2$. By Lemma \ref{lemC}, $n-t \le 3(\frac{n-6}{5})+4=\frac{3n+2}{5}$ (see (A) of Figure~\ref{mrupper}); thus $t \geq \frac{2n-2}{5}$. If $n$ is odd, noting that $|S \cap V(G_2)| \ge 1$, $|S| \geq \left\lceil\frac{2n-2}{5}\right\rceil + 1 = \left\lceil\frac{2n+3}{5}\right\rceil$. If $n$ is even, by (a) of Lemma \ref{lemC}, $|S| \geq \left\lceil\frac{2n+3}{5}\right\rceil+1$. 

\vspace{.1in}

\emph{Case 2. There is no vertex $u \in V(G_1)$ such that $u \not \in N[S] \cap V(G_1)$:} In this case, $\{u_1, u_2, \ldots, u_{n}\} \subseteq N[S] \cap V(G_1)$. By Lemma \ref{lemC}, $n-t \le 3(\frac{n-5}{5})+3=\frac{3n}{5}$ (see (B) of Figure~\ref{mrupper}); thus $t \geq \frac{2n}{5}$. Since $|S \cap V(G_2)| \ge 1$, we have $|S| \geq \left\lceil\frac{2n}{5}\right\rceil + 1$. 

\vspace{.1in}

Since one of these two cases must occur, for $n$ odd, $\dim(C(C_n,f_0)) \ge  \min \{\left\lceil\frac{2n+3}{5}\right\rceil, \left\lceil\frac{2n+5}{5}\right\rceil\} =\left\lceil\frac{2n+3}{5}\right\rceil$. For $n$ even, $\dim(C(C_n,f_0)) \ge \min \{\left\lceil\frac{2n+3}{5}\right\rceil+1, \left\lceil\frac{2n+5}{5}\right\rceil\}=\left\lceil\frac{2n+5}{5}\right\rceil$. $\Box$

\begin{figure}[htpb]
\begin{center}
\scalebox{0.45}{\input{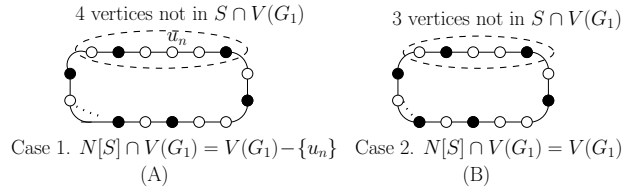}} \caption{The solid vertices in diagram (A) (resp., (B)) form $S \cap V(G_1)$ in Case 1 (resp., Case 2). }\label{mrupper}
\end{center}
\end{figure}

\textbf{Claim 2.} For $n \ge 6$, $|S| \le \lceil \frac{2n+3}{5}\rceil$ if $n$ is odd and $|S| \le \lceil \frac{2n}{5}\rceil+1$ if $n$ is even.\\

\textit{Proof of Claim 2.} We will show the existence of a resolving set $S$ of $C(C_n, f_0)$ of cardinalities given by the bounds. We consider two cases.

\vspace{.1in}

\emph{Case 1. $n \equiv 0,2,4  \pmod 5$:} If $n \equiv 0 \pmod 5$, then $S=\{u_{5i+2}, u_{5i+5} \mid 0 \le i \le \frac{n-5}{5}\} \cup \{v_2\}$ is a resolving set of $C(C_n, f_0)$; if $n \equiv 2 \pmod 5$, then $S=\{u_{5i+2}, u_{5i+5} \mid 0 \le i \le \frac{n-7}{5}\} \cup \{u_n\} \cup \{v_2\}$ is a resolving set of $C(C_n, f_0)$; if $n \equiv 4 \pmod 5$, then $S=\{u_{5i+2}, u_{5i+5} \mid 0 \le i \le \frac{n-9}{5}\} \cup \{u_{n-2}, u_n\} \cup \{v_2\}$ is a resolving set of $C(C_n, f_0)$. Since $N[S] \cap V(G_1)=V(G_1)$ and $|S \cap V(G_1)|  \ge  3$ for $n \ge 6$, by (b) and (c) of Lemma \ref{lemC}, $S \cap V(G_1)$ resolves all vertices in $G_1$ and no vertex in $G_1$ has the same code with a vertex in $G_2$. Further, a vertex in $S \cap V(G_1)$ and $v_2$ resolves all vertices in $G_2$. Thus, $S$ is a resolving set of $C(C_n, f_0)$ with $|S|=\lceil\frac{2n}{5}\rceil+1=\lceil\frac{2n+3}{5}\rceil$.

\vspace{.1in}

\emph{Case 2. $n \equiv 1,3 \pmod 5$:} If $n$ is odd and $n \equiv 1 \pmod 5$, then $S=\{u_2, u_{n-2}\} \cup \{u_{5i+4}, u_{5i+7} \mid 0 \le i \le \frac{n-11}{5}\} \cup \{v_{\frac{n+1}{2}}\}$ is a resolving set of $C(C_n, f_0)$ with $|S|=\frac{2n+3}{5}=\lceil\frac{2n+3}{5}\rceil$ (see (A) of Figure \ref{functiC11}); if $n$ is odd and $n \equiv 3 \pmod 5$, then $S=\{u_2, u_{n-4}, u_{n-2}\} \cup \{u_{5i+4}, u_{5i+7} \mid 0 \le i \le \frac{n-13}{5}\} \cup \{v_{\frac{n+1}{2}}\}$ is a resolving set of $C(C_n, f_0)$ with $|S|=\frac{2n+4}{5}=\lceil\frac{2n+3}{5}\rceil$. If $n$ is even and $n \equiv 1 \pmod 5$, then $S=\{u_2, u_{n-2}\} \cup \{u_{5i+4}, u_{5i+7} \mid 0 \le i \le \frac{n-11}{5}\} \cup \{v_{\frac{n}{2}}, v_{\frac{n}{2}+1}\}$ is a resolving set of $C(C_n, f_0)$ with $|S|=\frac{2n+3}{5}+1= \frac{2n+8}{5}=\lceil\frac{2n}{5}\rceil+1$ (see (B) of Figure \ref{functiC11}); if $n$ is even and $n \equiv 3 \pmod 5$, then $S=\{u_2, u_{n-4}, u_{n-2}\} \cup \{u_{5i+4}, u_{5i+7} \mid 0 \le i \le \frac{n-13}{5}\} \cup \{v_{\frac{n}{2}}, v_{\frac{n}{2}+1}\}$ is a resolving set of $C(C_n, f_0)$ with $|S|=\frac{2n+4}{5}+1= \frac{2n+9}{5}=\lceil\frac{2n}{5}\rceil+1$. For each case, noting that $\{u_n\}=V(G_1) - N[S]$, by (b) and (c) of Lemma \ref{lemC}, $S_1=S \cap V(G_1)$ resolves all vertices but $u_n$ in $G_1$ and $code_{S_1}(u_n)=code_{S_1}(v_2)=code_{S_1}(v_n)$. If $n$ is odd, a vertex in $S \cap V(G_1)$ and $v_{\frac{n+1}{2}}$ resolve all vertices in $G_2$; further, noting that $d(u_n, v_{\frac{n+1}{2}})=1+d(v_1, v_{\frac{n+1}{2}})$, we have $d(u_n, v_{\frac{n+1}{2}})> \max\{d(v_2, v_{\frac{n+1}{2}}), d(v_n, v_{\frac{n+1}{2}})\}$, and hence $code_S(u_n) \neq code_S(v_2)$ and $code_S(u_n) \neq code_S(v_n)$. If $n$ is even, a vertex in $S \cap V(G_1)$ and $\{v_{\frac{n}{2}}, v_{\frac{n}{2}+1}\}$ resolve all vertices in $G_2$; further, noting that $d(u_n, v_{\frac{n}{2}})=1+d(v_1, v_{\frac{n}{2}})> d(v_2, v_{\frac{n}{2}})$ and $d(u_n, v_{\frac{n}{2}+1})=1+d(v_1, v_{\frac{n}{2}+1})> d(v_n, v_{\frac{n}{2}+1})$, we have $code_S(u_n) \neq code_S(v_2)$ and $code_S(u_n) \neq code_S(v_n)$. Thus, $S$ is a resolving set of $C(C_n, f_0)$.

\begin{figure}[htpb]
\begin{center}
\scalebox{0.43}{\input{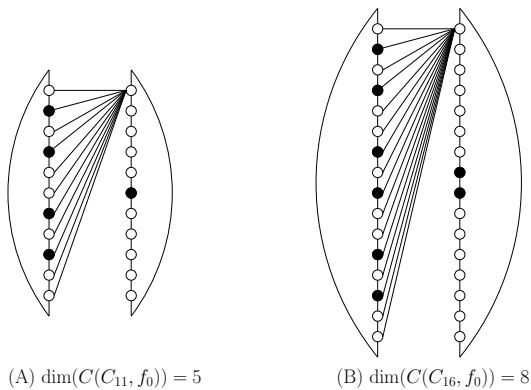}} \caption{The metric dimension of $C(C_{11}, f_0)$ and $C(C_{16}, f_0)$ and their minimum resolving sets, where the black vertices in each functigraph form a minimum resolving set}\label{functiC11}
\end{center}
\end{figure}

Therefore, for $n \ge 6$, by Claim 1 and Claim 2, we have $\dim(C(C_n,f_0))= \lceil \frac{2n+3}{5}\rceil$ if $n$ is odd and $\dim(C(C_n, f_0))= \lceil \frac{2n}{5}\rceil+1$ if $n$ is even.\hfill
\end{proof}

\begin{Thm}\label{functiCbetween}
For $n \ge 3$, let $G=C_n$. Let $f: V(G_1) \rightarrow V(G_2)$ be any function with $|f(V(G_1))|=s$, where $1 < s < n$. Then $2 \le \dim(C(C_n, f)) \le 2(n-1)-s$. 
\end{Thm}

\begin{proof}
By Theorem \ref{functibounds}, $\dim(C(C_n, f)) \ge 2$. Let $W=f(V(G_1))$. We will show that $\dim(C(C_n,f)) \le 2(n-1)-|W|$. We consider two cases.\\

\emph{Case 1. $|W|=n-1$:} Without loss of generality, we may assume that $|f^{-1}(v_1)|=2$, say $f^{-1}(v_1)=\{u_1, u_x\}$, by relabeling if necessary. Let $v_y \in V(G_2) - W$. We will show that $S=\{u_2, u_3, \ldots, u_n\}$ is a resolving set of $C(C_n, f)$. Note that (i) $u_1$ has 1 exactly in the 1st and in the $(n-1)$th entries of its code; (ii) $f(u_1)$ has 1 exactly in the $(x-1)$-th entry of its code for some $x$ ($2 \le x \le n$) such that $u_x \in f^{-1}(v_1)$; (iii) for $2 \le i \le n$, each $f(u_i)$ has 1 exactly in the $(i-1)$th entry of its code; (iv) $v_y$ does not contain 1 in any entry of its code. Since $S$ is a resolving set of $C(C_n, f)$ with $|S|=n-1$, $\dim(C(C_n,f)) \le n-1=2(n-1)-|W|$.\\

\emph{Case 2. $2 \le |W| \le n-2$:} In this case, $n \ge 4$. For $C(C_4, f)$ with $|f(V(G_1))|=2$, there are six non-isomorphic graphs (see Figure \ref{FC4}); one can readily check that $\dim(C(C_4, f ))\le 4$ for each case. 
\begin{figure}[htpb]
\begin{center}
\scalebox{0.41}{\input{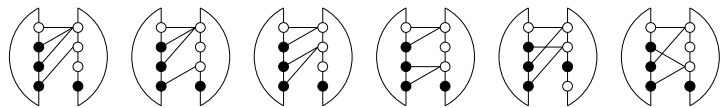}} \caption{Six non-isomorphic $C(C_4, f)$ with $|f(V(G_1))|=2$, where the black vertices in each functigraph form a resolving set}\label{FC4}
\end{center}
\end{figure}

So, let $n \ge 5$, and we consider two subcases. 

\vspace{.1in}

\emph{Subcase 2.1. There exists a vertex, say $u_1$, in $G_1$ with $N(u_1) \cap V(G_1)=\{u_2, u_n\}$ such that $f(u_2) \neq f(u_n)$:} Let $S_2=V(G_2) - (W \cup \{v'\})$ for some $v' \in V(G_2) - W$; we note here that $S_2 \neq \emptyset$. Since $|V(G_2) - W| \ge 2$ and $n \ge 5$, there exists a vertex, say $v' \not\in W$, in $G_2$ such that $N(v')\cap V(G_2) \neq \{f(u_2), f(u_n)\}$. We will show that $S=S_1 \cup  S_2$ is a resolving set of $C(C_n,f)$, where $S_1=\{u_2, u_3, \cdots, u_n\}$ and $S_2=\{v_1, v_2, \ldots, v_n\} - (W \cup \{v'\})$. Note that, for the 1st through the $(n-1)$th entries of its code, (i) $u_1$ has 1 exactly in the 1st and in the $(n-1)$th entries of its code; (ii) if $|f^{-1}(f(u_1))|=1$, then $f(u_1)$ does not contain 1 but has 2 in the 1st and in the $(n-1)$th entries of its code; if $|f^{-1}(f(u_1))| \ge 2$, then see (iii); (iii) for $2 \le i \le n$, each $f(u_i)$ has 1 in the $(i-1)$th entry of its code, but no $f(u_i)$ has 1 in the 1st and in the $(n-1)$th entries of its code at the same time; further, there's exactly one vertex in $W \subset V(G_2)$ with 1 in the $(i-1)$th entry of its code; (iv) $v' \in V(G_2)$ does not contain 1 in the 1st through $(n-1)$th entries of its code, and $v'$ does not have 2 in the 1st and in the $(n-1)$th entries of its code at the same time either. Since $S$ is a resolving set of $C(C_n, f)$ with $|S|=(n-1)+(n-s-1)=2(n-1)-s$, $\dim(C(C_n,f)) \le 2(n-1)-s$.

\vspace{.1in}

\emph{Subcase 2.2. For each vertex, say $u$, in $G_1$, two vertices in $N(u) \cap V(G_1)$ are mapped to the same vertex in $G_2$:} Notice that $|W|=1$ if $C_n$ is an odd cycle, and $|W|=2$ if $C_n$ is an even cycle. Since $|W| \ge 2$, $C_n$ must be an even cycle and $f(u_1)=f(u_3)=\cdots=f(u_{n-1})=v'$ and $f(u_2)=f(u_4)=\cdots=f(u_n)=v''$ for $v' \neq v''$. In this case, $S=S_1 \cup  S_2$ is a resolving set of $C(C_n, f)$, where $S_1 =\{u_3, u_4, \ldots, u_n\}$ and $S_2=\{v_1, v_2, \ldots, v_n\} - \{v', v''\}$. Note that, for the 1st through the $(n-2)$th entries of its code, (i) $u_1$ has 1 only in the $(n-2)$th entry; (ii) $u_2$ has 1 only in the $1$st entry; (iii) $v'$ has 1 in the $\ell_1$-th entry, where $1 \le \ell_1 \le n-2$ and $\ell_1$ is odd; (iv) $v''$ has 1 in the $\ell_2$-th entry, where $1 \le \ell_2 \le n-2$ and $\ell_2$ is even. Since  $S$ is a resolving set of $C(C_n, f)$, $\dim(C(C_n, f)) \le 2n-4=2(n-1)-|W|$. \hfill
\end{proof}

\begin{Rem}
For the graph in Figure~\ref{functiC=}, where $n=3$ and $|f(V(G_1))|=2$, the formula for the upper bound in Theorem~\ref{functiCbetween} yields $2$, which is also the lower bound. 
Notice that $\{u_2, u_3\}$ is a minimum resolving set of $C(C_3,f)$. 
\end{Rem}

\begin{figure}[htpb]
\begin{center}
\scalebox{0.45}{\input{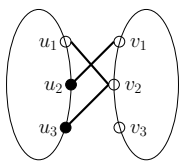}} \caption{A functigraph showing the sharpness of the bounds in Theorem~\ref{functiCbetween}} \label{functiC=}
\end{center}
\end{figure}

\textit{Acknowledgement.} The authors thank the anonymous referees for some helpful comments and suggestions.

\end{document}